\documentclass[]{amsart}
\usepackage{graphicx}

\newcommand{\R}{\mathbb{R}}

\newcommand{\Z}{\mathbb{Z}}

\newtheorem{thm}{Theorem}[section]
\newtheorem{lemma}[thm]{Lemma}

\newtheorem{cor}[thm]{Corollary}


\title{Further Study on Domains and Quasihyperbolic Distances}

\author{Shusen Ding, Dylan Helliwell, Gavin Pandya, Arya Yae}
\date{\today}
\thanks{Gavin Pandya and Arya Yae were supported by Seattle University's College of Science and Engineering Summer Undergraduate Research program in 2018.}

\begin{document}

\maketitle

\begin{abstract}
We establish constructive geometric tools for determining when a domain is $L^s$-averaging and obtain upper and lower bounds for the $L^s$-integrals of the quasihyperbolic distance. We also construct examples which are helpful to understand our geometric tools and the relationship between $p$-Poincar\'{e} domains and $L^s$-averaging domains.  Finally, finite unions of $L^s(\mu)$-averaging domains are explored.
\end{abstract}

\section{Introduction}

Domains and mappings are fundamental objects which have been well studied and applied in many fields of mathematics and engineering, including partial differential equations, potential analysis and harmonic analysis.
It is well known that domains affect the properties of objects defined on them such as functions, mappings, differential forms, integrals, and differential equations.  There is are a number of analytic criteria that can be used to classify various domains in $\mathbb{R}^n$, such as uniform domains, John domains and $L^s$-averaging domains and a typical goal is to determine the relationships among these criteria.  The quasihyperbolic distance provides a powerful tool which has been widely used in geometric analysis in recent years, for example, to characterize $L^s$-averaging domains and $L^s(\mu)$-averaging domains.  In this paper, we provide constructive geometric tools for determining when this characterization is met.

This paper is organized as follows.  After introducing notation and background information in Section \ref{backgroundsec}, we then define essential tubes and provide basic examples in Section \ref{tubesec}.  In Section \ref{Whitneysec}, the notion of a generalized Whitney subdivision is introduced and basic properties are established.  In Sections \ref{cuspsec} and \ref{blockdomainsec}, essential tubes and the idea of generalized Whitney subdivision are used to prove necessary and sufficient conditions for cusps and domains built using particular families of cubical blocks to be $L^s$-averaging.  Finally, in Section \ref{unionsec}, finite unions of $L^s(\mu)$-averaging domains are explored.

\section{Background} \label{backgroundsec}
In this section, after introducing some notation, we review the analytical criteria of interest in this paper, along with some of the known relationships among these criteria.

Generally, we will use $z$ to denote a point in a domain and we will reserve $x$ and $y$ for coordinates.

Throughout, we consider bounded and connected domains $\Omega$ of $\mathbb{R}^n$ with $n\geq2$.  For any set $E$ in $\mathbb{R}^n$, we denote by $|E|$ the Lebesgue measure of $E$, and for the purposes of integration, we will use $dz$ to denote the Lebesgue measure.  We will also consider more general measures $\mu$ defined in terms of a weight function $w$ so that $d\mu = w(z)dz$.  In these instances, the measure of a set $E$ will be denoted $\mu(E)$.  For a function $u \in L^1(\Omega)$, we denote the mean by ${u_{\Omega}}$.

We use a capital $C$ to indicate a positive constant, with optional arguments, such as the dimension $n$, to indicate on what a constant may depend.  This constant may be different in different instances.  Subscripts may be used when distinctions are necessary.

The following definition of $L^s$-averaging domains was introduced by Susan G. Staples in \cite{Staples}. For $1 \leq s <\infty$,  a domain $\Omega$ is called an \textit{$L^s$-averaging domain} if for all $u\in L^1_\text{loc}(\Omega, dz)$ it follows that
\begin{equation*}
	\left(\frac{1}{|\Omega|}\int_\Omega |u-{u_{\Omega}}|^s\,dz\right)^{1/s}
	\leq
	C(s, \Omega)\left(\sup_{ B\subset\Omega}\frac{1}{|B|}\int_B |u-{u_B}|^s\,dz\right)^{1/s},
\end{equation*}
where $B$  is any open ball in $\Omega$.

Many results about differential forms and related operators were established in $L^s$-averaging domains, see for example \cite{AgDiNo, DingLiu, LiuDing}.  In \cite{DingNolder}, $L^s$-averaging domains were extended to weighted averaging domains, $L^s(\mu)$-averaging domains, and a characterization in terms of Whitney cubes was provided.  Generalizing further, in \cite{Ding}, $L^\varphi(\mu)$-averaging domains were considered, where $\varphi$ is a convex function defined on $(0,\infty)$.

The following definition of the quasihyperbolic distance can be found in \cite{GehringOsgood}. For any points $z$ and $z_0$ in $\Omega$, the quasihyperbolic distance between $z$ and $z_0$ is given by
\begin{equation*}
k(z,z_0;\Omega)=
	\inf_{\gamma \subset \Omega}
	    \int_\gamma \frac{1}{d(\zeta,\partial \Omega)} d\sigma
	= \inf_{\gamma \subset \Omega}
	    \int_I \frac{|\gamma'(t)|}{d(\gamma(t),\partial\Omega)}\,dt
\end{equation*}
where $\gamma:I \rightarrow \Omega$ is a rectifiable curve connecting $z$ to $z_0$ and the infimum is being taken over all such curves.  F. W. Gehring and B. Osgood \cite{GehringOsgood} proved that for any two points in $\Omega$ there is a quasihyperbolic geodesic arc joining them.  In \cite{Staples}, Staples showed that $\Omega$ is an $L^s$-averaging domain if and only if
\begin{equation*}
	\left(\frac{1}{|\Omega|}\int_\Omega k(z,z_0;\Omega)^s\,dz \right)^{1/s}\leq C,
\end{equation*}
where $z_0$ is any fixed point in $\Omega$ and $C$ is a constant depending only on $n$, $s$, $|\Omega|$, the choice of $z_0 \in \Omega$, and the constant from the inequality in the definition of $L^s$-averaging domains.  Using this characterization, it was also shown in \cite{Staples} that John domains are $L^s$-averaging for all $1 \leq s < \infty$.

For $1 \leq p < \infty$ we say a domain $\Omega$ a \textit{$p$-Poincar\'e Domain} if for every function $u$ in the Sobolev space $W^{1,p}(\Omega)$,
\begin{equation*}
	\|u-{u_{\Omega}}\|_{L^p(\Omega)} \leq C(p, \Omega) \|\nabla u\|_{L^p(\Omega)}.
\end{equation*}
In \cite{Staples} it was shown that for $p \geq n$, if $\Omega$ is $L^p$-averaging, then it is $p$-Poincar\'{e} as well.  Staples also showed, by an explicit example, that this relationship need not hold if $p < n$.  Specifically, she constructed a ``rooms-and-halls'' domain that was $L^s$-averaging for all $s \geq 1$, but was not $p$-Poincar\'{e} for any $p < n$.  Of course, by the previously mentioned result, this domain is necessarily $p$-Poincar\'{e} for $p \geq n$.

In somewhat of a contrast, in \cite{SmithSteg} it was shown that star-shaped domains are $p$-Poincar\'{e} for all $1 \leq p < \infty$, but as shown in \cite{Staples} and in Section \ref{cuspsec} and \ref{blockdomainsec} below, there are star-shaped domains that are $L^s$-averaging if and only if $1 \leq s < k$ where $k$ depends on the dimension and parameters defining the domain.

\section{Essential tubes} \label{tubesec}

In this section, we define essential tubes, use them to generate necessary conditions for domains to be $L^s$-averaging, and provide some examples of domains that are not $L^s$-averaging for any $1 \leq s < \infty$.
\subsection{Essential tubes defined}

Let $D_r^k \subset \mathbb{R}^k$ be the closed $k$-dimensional disk of radius $r$ centered at the origin.  Consider the cylinder $[0, l] \times D_r^{n-1}$ in $\mathbb{R}^n$.  We define a tube $T_{l,r}$ to be the image under a Euclidean transformation of this cylinder.  We say the images of $\{0\} \times D_r^{n-1}$ and $\{l\} \times D_r^{n-1}$ are the ends of the tube and we say the image of $[0,l] \times (\partial D_r^{n-1})$ is the wall of the tube.

Let $\Omega$ be a domain in $\mathbb{R}^n$.  We define an essential tube $T = T_{l, r, c}$ for $\Omega$ to be a tube $T_{l, r}$ such that $T \cap \Omega$ has a connected component $\Omega_T$ satisfying the following properties:
\begin{itemize}
    \item the intersection of $\Omega_T$ with the wall of $T$ is empty,
    \item There exists $c > 0$ such that for all $t \in [0, l]$, the $n-1$ dimensional measure of the $t^{\mathrm{th}}$ slice of $\Omega_T$ is at least $c$ times the measure of $D_r^{n-1}$.
\end{itemize}
See Figure \ref{essentialtubefig}.
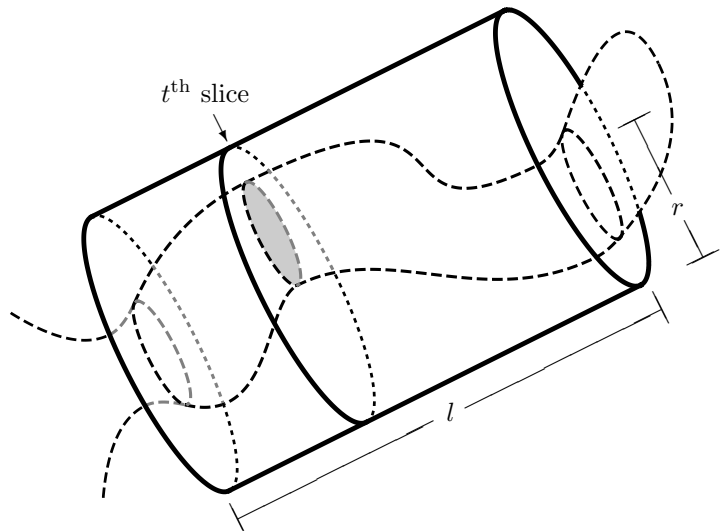
\begin{figure}
\begin{picture}(285,200)
\put(0,0){
\includegraphics[scale = 1, clip = true, draft = false]{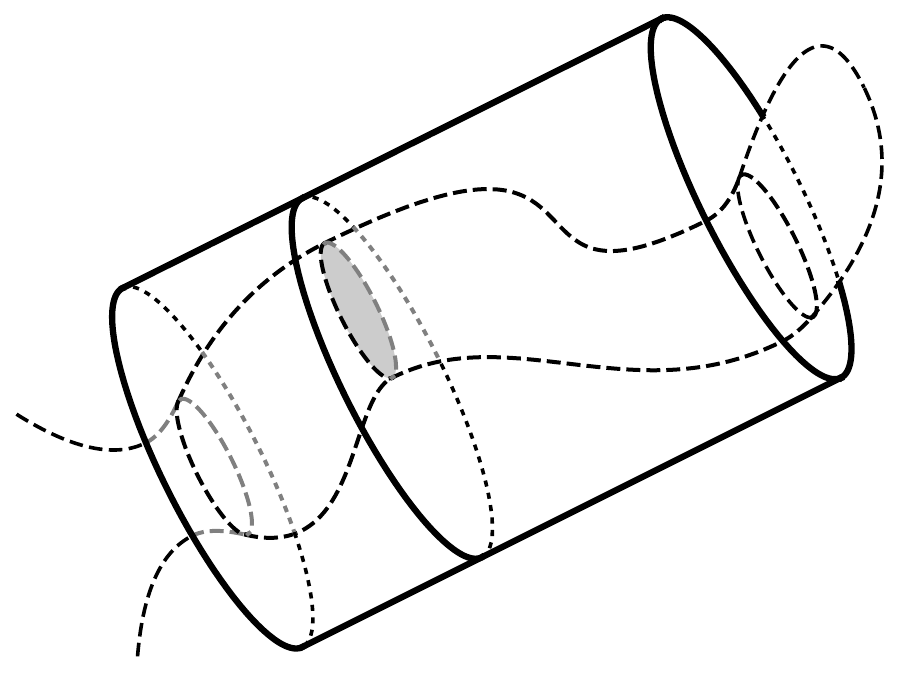}
}
\put(65, 155){$t^{\mathrm{th}}$ slice}
\put(85, 150){\vector(1,-2){5}}
\put(94,-3){\line(2,1){74}}
\put(180,39){\line(2,1){74}}
\put(91,3){\line(1,-2){6}}
\put(251,83){\line(1,-2){6}}
\put(173,34){$l$}
\put(238,145){\line(2,1){12}}
\put(264,93){\line(2,1){12}}
\put(244,148){\line(1,-2){14}}
\put(270,96){\line(-1,2){7}}
\put(258,112){$r$}
\end{picture}
\caption{An essential tube $T = T_{l,r,c}$ enclosing a portion $\Omega_T$ of a domain $\Omega$.  The shaded region is the $t^{\mathrm{th}}$ slice of $\Omega_T$ and its $n-1$ dimensional measure is at least $c$ times the measure of the corresponding slice of $T$.} \label{essentialtubefig}
\end{figure}

\subsection{Quasihyperbolic distance calculations}

First, we obtain the lower bound of the $L^s$-integral of the quasihyperbolic distance provided by $\Omega_T$ in the following theorem.

\begin{thm}
Let $\Omega$ be a domain, let $T = T_{l,r,c}$ be an essential tube for $\Omega$ with corresponding component $\Omega_T$, and let $z_0 \in \Omega$ be any point not in $\Omega_T$.  Then
\begin{equation*}
        \int_{\Omega_T} k(z, z_0; \Omega_T)^s\, dz
        \geq C(s,n) c r^{n}                 \left(\frac{l}{r}\right)^{s+1}.
\end{equation*}
\end{thm}

\begin{proof}
First, choose coordinates so that the wall of the tube aligns with the first coordinate and, writing $z = (x_1, x_2, \ldots, x_n)$, one end corresponds to $x_1 = 0$, and the other end corresponds to $x_1 = l$.  For any point $z \in \Omega_T$, let $\gamma:[a,b] \rightarrow \Omega$ be a rectifiable curve connecting $z$ to $z_0$.  Let $\gamma$ leave $z$ for the last time at time $\alpha$, and let $\gamma$ leave $\Omega_T$ for the first time at time $\beta$.  Then for all $t \in [\alpha, \beta]$, $\gamma(t) \in \Omega_T$, $d(\gamma(t), \partial \Omega_T) \leq r$,
and $|\gamma'(t)|\geq |\gamma_1'(t)|$.  Using these estimates, and accounting for the fact that the curve may leave either end of the tube,
\begin{align*}
    \int_{\gamma} \frac{1}{d(\zeta, \partial \Omega_T)}\, d\sigma
            &= \int_a^b \frac{|\gamma'(t)|}{d(\gamma(t), \partial \Omega_T)}\, dt \\
            &\geq \int_{\alpha}^{\beta} \frac{|\gamma_1'(t)|}{r}\, dt \\
            &\geq \min\left\{\frac{x_1}{r}, \frac{l-x_1}{r}\right\}.
\end{align*}
This is true for all rectifiable curves so for $z \in \Omega_T'= \Omega_T \cap \left\{z:x_1 \leq \frac{l}{2}\right\}$, 
\[
k(z, z_0; \Omega) \geq \frac{x_1}{r}.
\]
Hence, letting $V_k$ be the volume of the unit disk in $\R^k$,
\begin{align*}
    \int_{\Omega_T} k(z, z_0; \Omega)^s\, dz
        &\geq \int_{\Omega_T'}
        k(z, z_0; \Omega)^s\, dz \\
        &\geq \int_{\Omega_T'} \left(\frac{x_1}{r}\right)^s \, dz \\
        &\geq \frac{1}{r^s}c\, (V_{n-1}r^{n-1}) \int_0^{\frac{l}{2}} x_1^s \, dx_1 \\
        &= \frac{V_{n-1}}{(s+1)2^{s+1}}c \, r^{n}                 \left(\frac{l}{r}\right)^{s+1}. 
\end{align*}
\end{proof}
With this result in hand, essential tubes can be used to show when a given domain fails to be $L^s$-averaging.  To help with this, we introduce the following notation:  Given a family $\mathcal{T}$ of essential tubes $T$ with parameters $r_T$, $l_T$, and $c_T$, define $E_{\mathcal{T}}$ to be the following sum:
\[
E_{\mathcal{T}} = \sum_{T \in \mathcal{T}} c_T (r_T)^n \left(\frac{l_T}{r_T}\right)^{s+1}.
\]
We now have the following
\begin{cor} \label{tubeestcor}
Let $\mathcal{T}$ be a family of essential tubes for $\Omega$ such that the corresponding components $\Omega_T$ are pairwise disjoint.  Let the parameters of $T \in \mathcal{T}$ be $r_T$, $l_T$, and $c_T$.  Then if $E_{\mathcal{T}}$ is infinite, then $\Omega$ cannot be $L^s$-averaging. \end{cor}

The proof of this result is left to the reader.  We demonstrate how this can be used in the examples below and in later sections.

\subsection{Examples}
For the first example, we construct a ``rooms-and-halls'' domain $\Omega \subset \R^2$ that is not $p$-Poincar\'{e} for any $1 \leq p < \infty$.  First, define two sequences $x_j=1-1/2^j$ and $x'_j=x_j+1/2^{j+2}$ for $j\in\Z^+$, and set $x'_0=0$.  Next, define a sequence of ``rooms'' by
\begin{equation*}
	R_j=[x'_j, x_{j+1}]\times[0,1]\qquad \text{ for } j=0,1,2,\dots
\end{equation*}
and ``halls'' by
\begin{equation*}
	H_j=[x_j,x'_j]\times \left[0,\frac{1}{(j+1)!}\right]\qquad \text{ for } j=1,2,\dots.
\end{equation*}
Letting $f(x,y)=(-x,y)$, set $A=R_0\cup\left[\bigcup_{j\in\Z^+}(R_j\cup H_j)\right]$, and define $\Omega=\text{int}(A\cup f(A))$.  See Figure \ref{roomsnhallsfig}.

\begin{figure}
\begin{picture}(265,140)
\put(0,0){
\includegraphics[scale = 1, clip = true, draft = false]{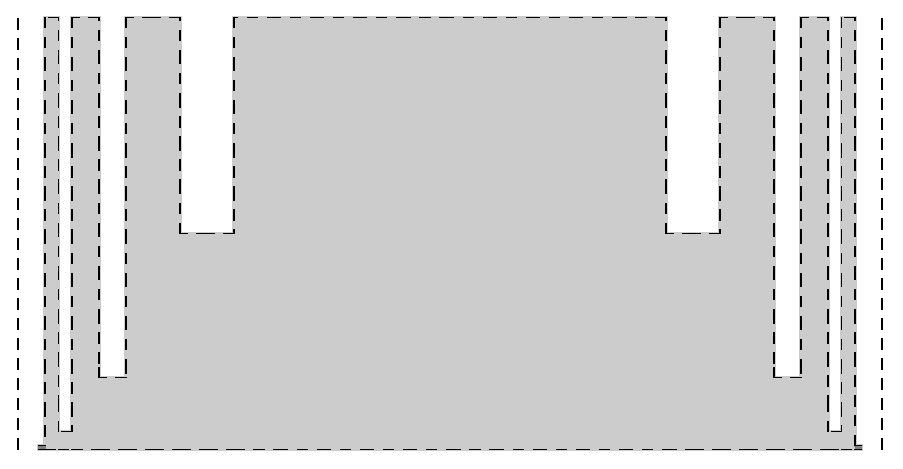}
}
\put(9.5,65){$\cdot \!\! \cdot \!\! \cdot$}
\put(251.5,65){$\cdot \!\! \cdot \!\! \cdot$}
\end{picture}
\caption{The rooms-and-halls domain.} \label{roomsnhallsfig}
\end{figure}

Now we construct a sequence of functions in $W^{1,p}(\Omega)$ that will demonstrate that $\Omega$ is not $p$-Poincar\'{e}.  Let
\begin{equation*}
	v_j(x)=
	\begin{cases}
		0&\text{ if }|x|<x_j\\
		2^{j+2}(x-x_j)&\text{ if }x_j\leq|x|\leq x'_j\\
		1&\text{ if }|x|>x'_j
	\end{cases}
\end{equation*}
for $j\in\Z^+$ and $x>0$, and let
\begin{equation*}
	u_j(x,y)=
	\begin{cases}
		v_j(x)&\text{ if }x\geq0\\
		-v_j(-x)&\text{ if }x<0
	\end{cases}
\end{equation*}
for $j\in\Z^+$.  Then $\{u_j\}_{j\in\Z^+}$ is the desired sequence of functions.  To see this, note that each $R_j$ and $H_j$ has width $1/2^{j+2}$.  Since ${u_j}_{\Omega}=0$,
\begin{equation*}
	\|u_j-{{u_j}_{\Omega}}\|_{L^p(\Omega)}^p
	    = \int_\Omega |u_j|^p\,dz
	    \geq 2\int_{R_j} |1|^p\,dz = 2\cdot\frac{1}{2^{j+2}}
\end{equation*}
and
\begin{equation*}
	\|\nabla u_j\|_{L^p(\Omega)}^p
	    = \int_\Omega |\nabla u_j|^p\,dz
	    = 2\int_{H_j} |2^{j+2}|^p\,dz
	    = 2\cdot(2^{j+2})^p
	        \left(\frac{1}{2^{j+2}}
	        \cdot\frac{1}{(j+1)!}\right).
\end{equation*}
Therefore,
\begin{equation*}
	a_j = \frac{\|u_j-{{u_j}_{\Omega}}\|_{L^p(\Omega)}}
	        {\|\nabla u_j\|_{L^p(\Omega)}}
	    \geq \left(\frac{(j+1)!}{(2^{j+2})^p}\right)^{\frac{1}{p}}
\end{equation*}
and this sequence diverges as $j\rightarrow\infty$ regardless of the choice of $p$.

Next, we use essential tubes to show that this rooms-and-halls domain is not $L^s$-averaging for any $1 \leq s < \infty$.  Note that the rectangles
$T_j = [x_j', x_{j+1}] \times \left[\frac{1}{2},\frac{3}{4}\right]$ are essential tubes for $j \geq 1$, and for each tube we have $r_j = \frac{1}{2^{j+3}}$, $l = \frac{1}{4}$, and $c = 1$.  Hence
\begin{align*}
E_{\{T_j\}} &= \sum_{j = 1}^{\infty} 1 \cdot \left(\frac{1}{2^{j+3}}\right)^2
\left(\frac{\frac{1}{4}}{\frac{1}{2^{j+3}}}\right)^{s+1} \\
    &= \frac{1}{4^{s+1}}
        \sum_{j = 2}^{\infty} (2^{s-1})^{j+3} \\
    &= \infty.
\end{align*}

Since the associated components $\Omega_{T_j}$ are pairwise disjoint, by Corollary \ref{tubeestcor}, $\Omega$ cannot be $L^s$-averaging.

Of course, this is not much of a surprise.  We already know that the rooms-and-halls domain is not $p$-Poincar\'{e}, so for $s \geq 2$ it cannot be $L^s$-averaging. The calculation above shows that it cannot be $L^s$-averaging for any $1 \leq s < \infty$.

The next, perhaps more interesting, example is a domain that is not $L^s$-averaging for any $s$, but is $p$-Poincar\'{e} for all $p$.  For $j\in\Z^+$, let $\theta_j = \left[1 - \left(\frac{1}{2}\right)^{j-1} \right] \pi$ and let $z_j = (\cos(\theta_j), \sin(\theta_j))$.  Let $R_j$ be the filled open rectangle with two vertices $z_j$ and $z_{j+1}$ and with the other two vertices lying on the circle of radius 3 centered at the origin.  Let $B$ be the open unit disk centered at the origin and define $\Omega$ to be the ``disk-and-rooms'' domain as follows:
\[
\Omega = \bigcup_{j = 1}^{\infty} R_j \cup B.
\]
See Figure \ref{diskandroomsfig}.
\begin{figure}
\begin{picture}(275,200)
\put(0,0){
\includegraphics[scale = 1, clip = true, draft = false]{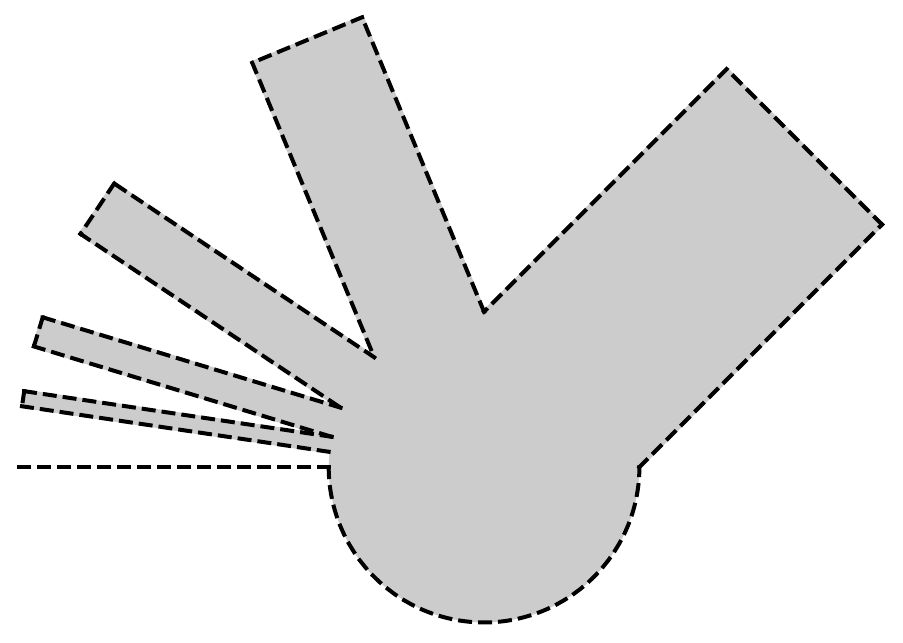}
}
\put(20,60){$\boldsymbol{\cdot}$}
\put(19,56){$\boldsymbol{\cdot}$}
\put(19,52){$\boldsymbol{\cdot}$}
\end{picture}
\caption{The disk-and-rooms domain.} \label{diskandroomsfig}
\end{figure}
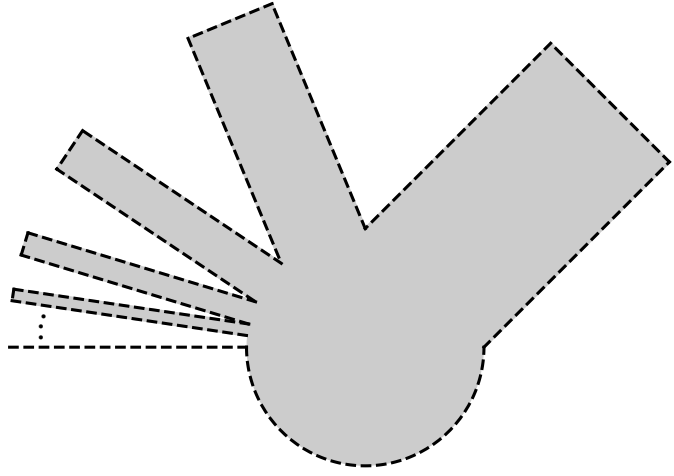
Note that $\Omega$ is star-shaped with respect to the origin.  Hence, it is a $p$-Poincar\'{e} domain for all $p$.

Let $T_j$ be the filled closed rectangle with two vertices $z_j$ and $z_{j+1}$ and with the other two vertices lying on the circle of radius 2 centered at the origin.  Then the $T_j$ are essential tubes and the associated sets $\Omega_{T_j}$ are pairwise disjoint.  The parameters of $T_j$ can be estimated as follows:  $r_j < \frac{\theta_{j+1} - \theta_j}{2} = \frac{\pi}{2^{j+1}}$, $l_j > 1$, and $c_j = 1$.

With this, we have
\begin{align*}
E_{\{T_j\}} &\geq \sum_{j = 1}^{\infty} 1 \cdot \left(\frac{\pi}{2^{j+1}}\right)^2
\left(\frac{1}{\frac{\pi}{2^{j+1}}}\right)^{s+1} \\
    &= \frac{1}{\pi^{s-1}}
        \sum_{j = 1}^{\infty} (2^{s-1})^{j+1} \\
    &= \infty
\end{align*}
and therefore, by Corollary \ref{tubeestcor}, $\Omega$ cannot be $L^s$-averaging.

\section{Generalized Whitney Subdivision} \label{Whitneysec}

In this section, we discuss a general method that can be used to establish sufficient conditions for a domain to be $L^s$-averaging.  In some ways this complements essential tubes, but this method is not as concrete.

Given a domain $\Omega$, we say a collection $\mathcal{S}$ of sets is a \textit{valid subdivision} if it has the following properties:
\begin{itemize}
    \item Each element $S \in \mathcal{S}$ is a closed subset of $\Omega$;
    \item Each element $S \in \mathcal{S}$ is star-shaped;
    \item For all distinct pairs $S, T \in \mathcal{S}$, $|S \cap T| = 0$;
    \item $\left|\Omega - \bigcup_{S \in \mathcal{S}} S \right| = 0$.
    \item For every pair of points $z_0, z \in \bigcup_{S \in \mathcal{S}} S$, there is a sequence $\{S_i: 0 \leq i \leq j\} \subset \mathcal{S}$ such that $z_0 \in S_0$, $z \in S_j$, and $\partial S_i \cap \partial S_{i+1} \neq \emptyset$.
\end{itemize}
For each $S \in \mathcal{S}$, define two parameters:  let $d(S)$ be the diameter of $S$ and let $\delta(S)$ be the distance between $S$ and $\partial \Omega$.  We then have the following:
\begin{lemma} \label{kinitialestlemma}
Let $\mathcal{S}$ be a valid subdivision for $\Omega$ and let $z_0$ and $z$ be two points in $\bigcup_{S \in \mathcal{S}} S$.  Let $\{S_i: 0 \leq i \leq j\} \subset \mathcal{S}$ be a sequence of sets such that $z_0 \in S_0$, $z \in S_j$, and $\partial S_i \cap \partial S_{i+1} \neq \emptyset$.  Then
\[
k(z, z_0; \Omega) \leq 2\sum_{i = 0}^{j} \frac{d(S_i)}{\delta(S_i)}.
\]
\end{lemma}

\begin{proof}
For $i \in \{0, \ldots, j\}$ let $\hat{z}_i \in S_i$ be a point relative to which $S_i$ is star-shaped.  For $i \in \{1, \ldots, j\}$, let $z_i \in S_{i-1} \cap S_i$, and let $z_{j+1} = z$.  For each $i \in \{0, \ldots, j\}$ let $\tilde{\gamma}_i:\tilde{I_i} \rightarrow S_i$ be a piecewise linear path connecting $z_i$ to $\hat{z}_i$ and then to $z_{i+1}$.  Note that in each $S_i$, $\tilde{\gamma}_i$ consists of two segments, both of which have length at most $d(S_i)$.  Then let $\tilde{\gamma}:\tilde{I} \rightarrow \Omega$ be the concatination of these $\tilde{\gamma}_i$.  See Figure \ref{validsubdivisionpathfig}.
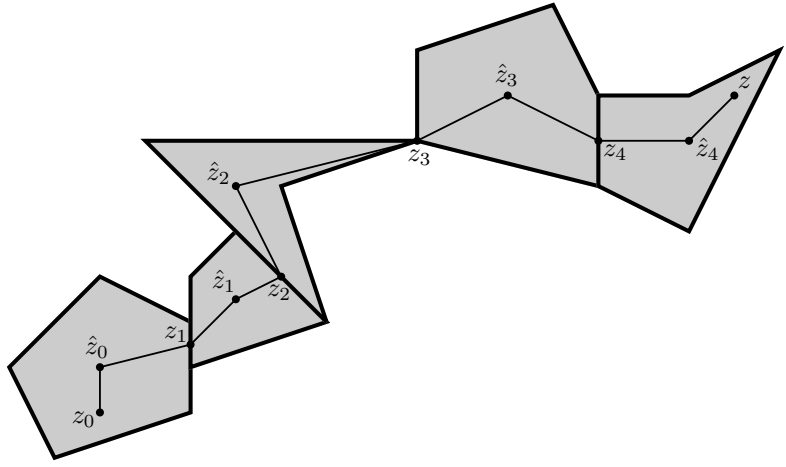
\begin{figure}
\begin{picture}(300,180)
\put(0,0){
\includegraphics[scale = .6, clip = true, draft = false]{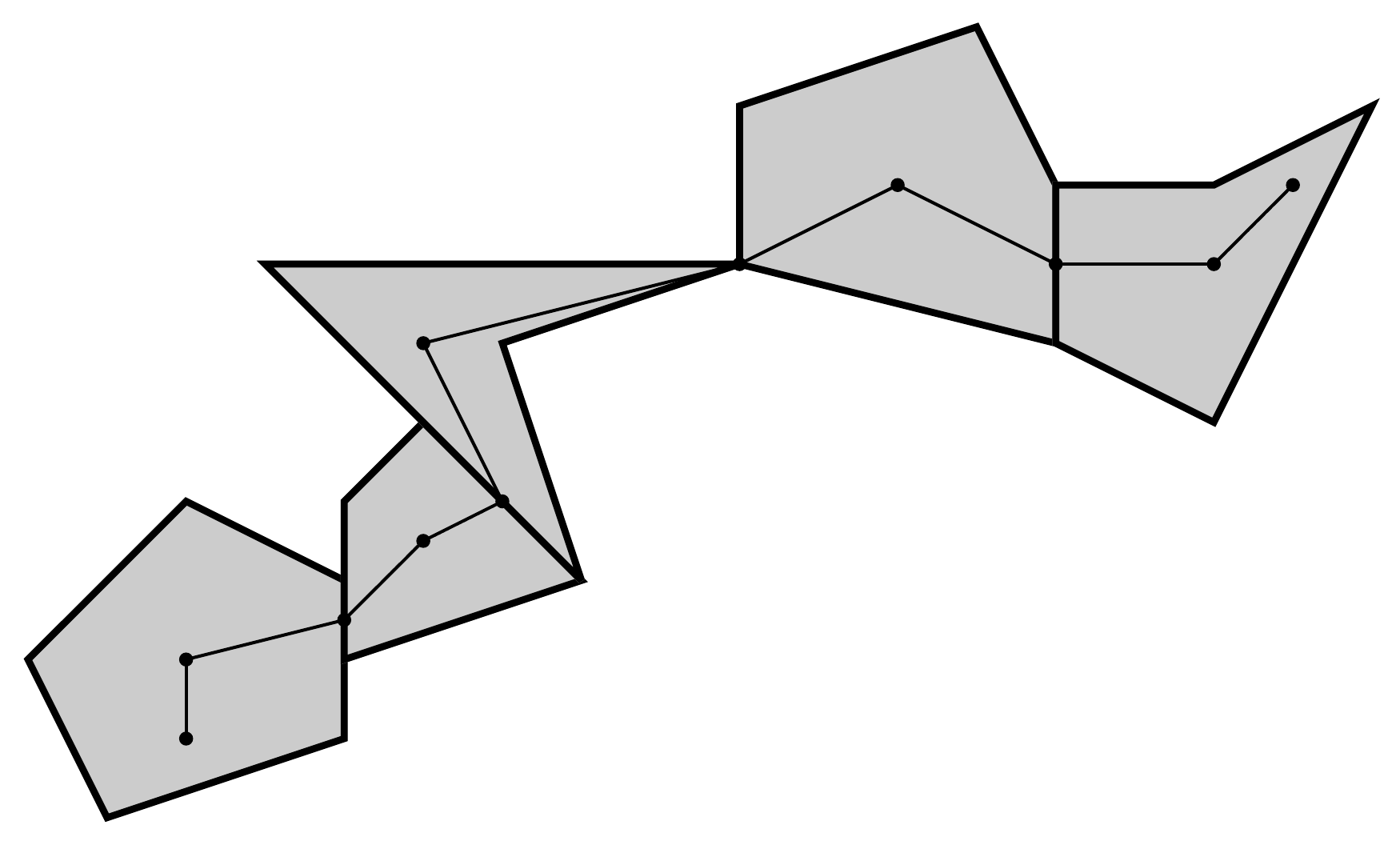}
}
\put(33,19){$z_0$}
\put(38,45){$\hat{z}_0$}
\put(68,51){$z_1$}
\put(86,70){$\hat{z}_1$}
\put(107,67){$z_2$}
\put(84,111){$\hat{z}_2$}
\put(160,118){$z_3$}
\put(193,148){$\hat{z}_3$}
\put(234,120){$z_4$}
\put(269,120){$\hat{z}_4$}
\put(285,146){$z$}

\end{picture}
\caption{Sets $S_0, \ldots, S_4$ in a valid subdivision, reference points, and path used for proof of Lemma \ref{kinitialestlemma}.} \label{validsubdivisionpathfig}
\end{figure}

This path provides the following estimate for $k(z, z_0; \Omega)$:
\begin{align*}
    k(z, z_0; \Omega)
        &= \inf_{\gamma} \int_{I_{\gamma}} \frac{|\gamma'(t)|}{d(\gamma(t), \partial D)} dt \\
        & \leq \int_{\tilde{I}} \frac{|\tilde{\gamma}'(t)|}{d(\tilde{\gamma}(t), \partial D)} dt \\
        &= \sum_{i = 0}^j \int_{\tilde{I_i}} \frac{|\tilde{\gamma}_i'(t)|}{d(\tilde{\gamma}_i(t), \partial D)} dt \\
        & \leq \sum_{i = 0}^j \int_{\tilde{I_i}} \frac{|\tilde{\gamma}_i'(t)|}{\delta(S_i)} dt \\
        &= \sum_{i = 0}^j \frac{1}{\delta(S_i)} \int_{\tilde{I_i}} |\tilde{\gamma}_i'(t)| dt \\
        & \leq \sum_{i = 0}^j \frac{1}{\delta(S_i)} 2 d(S_i)
\end{align*}
\end{proof}

The estimate in the above lemma is useful when the relationship between $d(S)$ and $\delta(S)$ is well behaved.  With this in mind, we say a valid subdivision $\mathcal{S}$ is a \textit{generalized Whitney subdivision} if there exists an $M$ such that for each $S \in \mathcal{S}$, $d(S) \leq M\delta(S)$.  We call $M$ the \textit{distance factor}.

For a generalized Whitney subdivision $\mathcal{S}$ of $\Omega$, with $z_0 \in S_0$, let $L_j$ be the union of those sets $S$ which are $j$ sets away from $S_0$ (so $L_0 = S_0$).  Then, we have the following upper bound for the $L^s$-integral of the quasihyperbolic distance:
\begin{lemma} \label{kestlemma}
Let $\mathcal{S}$ be a generalized Whitney subdivision for a domain $\Omega$ with distance factor $M$.  Let $z_0 \in S_0$.  Then for $z \in L_j$, $k_{\Omega}(z, z_0) \leq 2M(j+1)$ and
\[
\int_{\Omega} [k(z,z_0; \Omega)]^s\, dz
    \leq (2M)^s \sum_{j=0}^{\infty}(j+1)^s |L_{j}|.
\]
\end{lemma}

\begin{proof}
Using Lemma \ref{kinitialestlemma},
\begin{align*}
    k(z, z_0; \Omega)
        &\leq 2\sum_{i = 0}^{j} \frac{d(S_i)}{\delta(S_i)} \\
        &\leq 2 \sum_{i = 0}^{j} M \\
        &= 2M(j+1)
\end{align*}
and so
\begin{align*}
    \int_{\Omega} [k(z,z_0; \Omega)]^s\, dz
    &= \sum_{j=0}^{\infty}
        \int_{L_j} [k(z,z_0; \Omega)]^s\, dz \\
    &\leq \sum_{j=0}^{\infty}
        \int_{L_j}[2M(j+1)]^s\, dz \\
    &= (2M)^s \sum_{j=0}^{\infty} (j+1)^s
        \int_{L_j} dz \\
    &= (2M)^s \sum_{j=0}^{\infty} (j+1)^s |L_j|.
\end{align*}
\end{proof}

To demonstrate how such an estimate can be used, we show that a cube is $L^s$-averaging for all $s$.  This result is, of course, not new, but it illustrates how the analysis can be performed, and provides an upper bound to be used later.

\begin{lemma} \label{cubelemma}
Let $\Omega$ be the unit cube in $\mathbb{R}^n$ and let $z_0$ be the center point of $\Omega$.  Then
\begin{equation} \label{ksquareesteqn}
\int_{\Omega}[k(z, z_0; \Omega)]^s\, dz
    \leq C(n,s) \sum_{j=0}^{\infty}(j+1)^s \left(\frac{1}{3}\right)^j
\end{equation}
which is finite for all $s \geq 1$.
\end{lemma}

\begin{proof}
Let $L_0 = S_0$ be the closed cube of side length $\frac{1}{2}$ centered at $z_0$.  After this, to produce the $j$th layer of cubes, subdivide each exposed $n-1$ dimensional face of the cubes in the $(j-1)$st layer into $3^{n-1}$ congruent square pieces and let these be the faces of a new set of cubes of side length $\frac{1}{2}\frac{1}{3^j}$.  These cubes enclose most of the $(j-1)$st layer, but there are still lower dimensional components that are accessible.  With this in mind, complete the layer by adding more cubes of the same size so as to completely enclose the $(j-1)$st layer.  See Figure \ref{squarewhitneysubfig}.

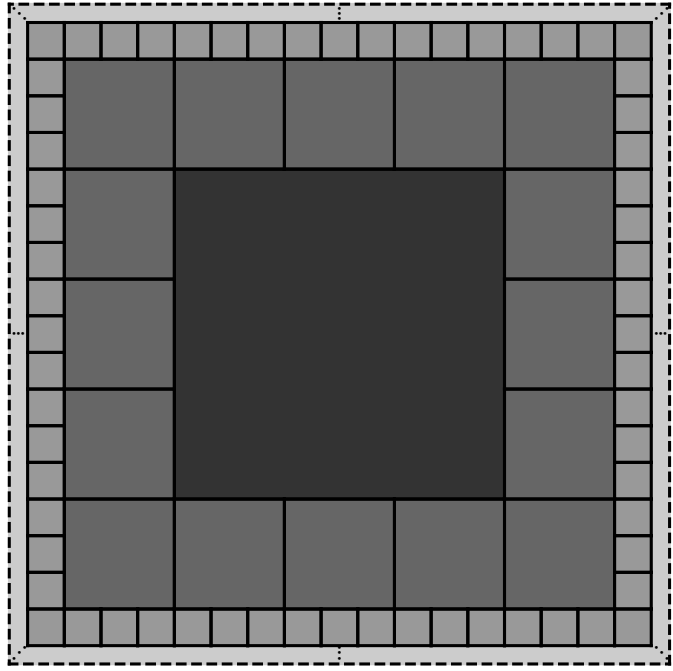
\begin{figure}
\begin{picture}(265,265)
\put(0,0){
\includegraphics[scale = 1, clip = true, draft = false]{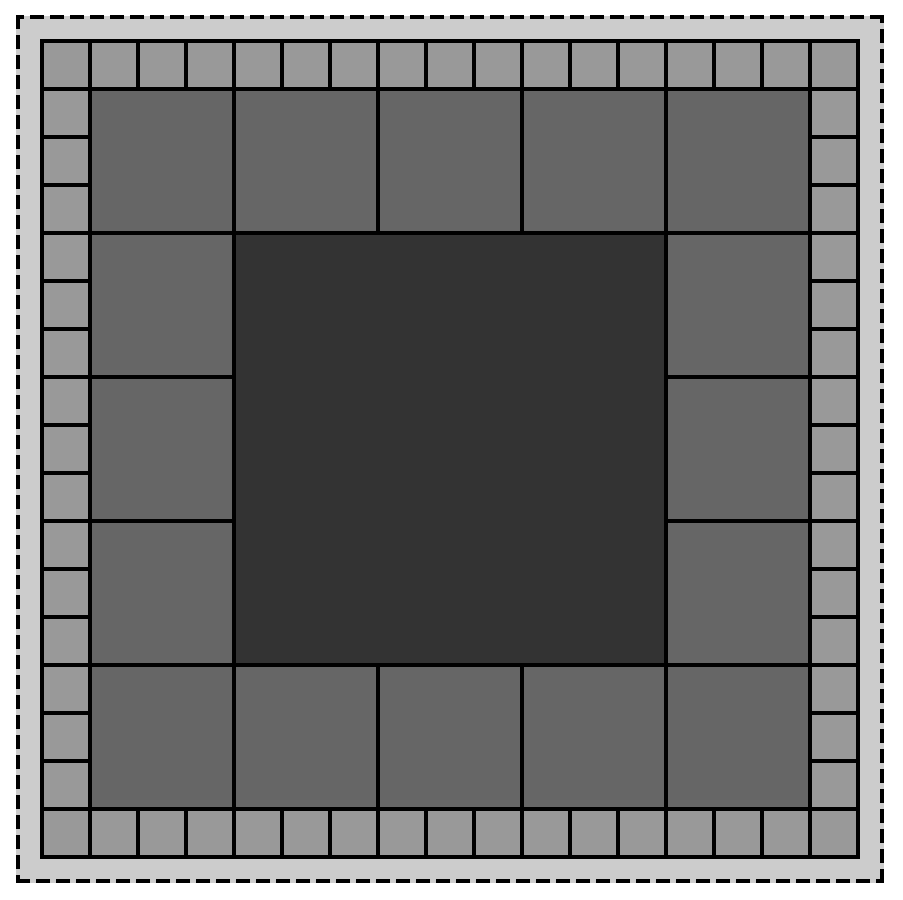}
}
\put(9,128){$\cdot \!\! \cdot \!\! \cdot$}
\put(252,128){$\cdot \!\! \cdot \!\! \cdot$}
\put(132,247){$\cdot$}
\put(132,249){$\cdot$}
\put(132,251){$\cdot$}
\put(132,9){$\cdot$}
\put(132,7){$\cdot$}
\put(132,5){$\cdot$}
\put(252,9){$\cdot$}
\put(254,7){$\cdot$}
\put(256,5){$\cdot$}
\put(252,247){$\cdot$}
\put(254,249){$\cdot$}
\put(256,251){$\cdot$}
\put(13,247){$\cdot$}
\put(11,249){$\cdot$}
\put(9,251){$\cdot$}
\put(13,9){$\cdot$}
\put(11,7){$\cdot$}
\put(9,5){$\cdot$}
\end{picture}
\caption{Generalized Whitney subdivision of a square.  The first three layers $L_0$, $L_1$, and $L_2$ are shown in increasingly lighter shades of gray.} \label{squarewhitneysubfig}
\end{figure}

As a union of small cubes, these layers are hollow cubes.  Let $e_i$ be the number of cubes along a one dimensional edge.  Then $e_0 = 1$ and $e_j = 3 e_{j-1}+2$, and from this we can conclude that $e_j = 2 \cdot 3^j - 1$.

Let $\nu_j$ be the number of cubes creating $L_j$.  Then $\nu_0 = 1$, and since an $n$-dimensional cube has $2n$ faces, for $j > 1$
\begin{align*}
\nu_j &< 2n e_j^{n-1} \\
    &= 2n (2 \cdot 3^j - 1)^{n-1} \\
    &< 2n (2 \cdot 3^j)^{n-1} \\
    &= 2^n n 3^{j(n-1)}
\end{align*}
where the first inequality comes from the fact that we are over counting the cubes included to cover the lower dimensional edges.

Since the side length of each cube $S$ in $L_j$ is $\frac{1}{2} \left(\frac{1}{3}\right)^j$, the diameter is $d(S) = \frac{\sqrt{n}}{2}\frac{1}{3^j}$, and for $j \geq 1$
\begin{align*}
    \delta(S) &= \frac{1}{4} - \sum_{i=1}^j \frac{1}{2} \left(\frac{1}{3}\right)^i \\
        &= \frac{1}{4} - \frac{1}{4}\left(1-\frac{1}{3^j}\right) \\
        &= \frac{1}{4 \cdot 3^j}.
\end{align*}
Hence, $d(S) =2 \sqrt{n}\, \delta(S)$ and so this is a generalized Whitney subdivision with a distance factor of $2 \sqrt{n}$.

Estimating the measure of $L_j$, we have
\begin{align*}
    |L_j| &= \nu_j             \left(\frac{1}{2\cdot3^j}\right)^n \\
    &\leq 2^n n 3^{j(n-1)} \frac{1}{2^n}\frac{1}{3^{jn}} \\
    &= n \frac{1}{3^j}. \\
\end{align*}

Applying Lemma \ref{kestlemma} we have Equation \eqref{ksquareesteqn} which converges for all $s$.
\end{proof}

For a given domain, cubes may not be an ideal object for subdivision, and other shapes, tailored to the domain, can be used.  An example is provided in the next section.

\section{Cusps} \label{cuspsec}

As a family, cusps demonstrate that a domain can be $L^s$-averaging for some $s$ and not others.  The cusps analyzed here were explored in \cite{Staples}.  We confirm those results using essential tubes and generalized Whitney subdivision.

\begin{thm} \label{cuspthm}
For $\alpha \geq 0$, let $\Omega_{\alpha} \in \mathbb{R}^n$ be the following domain:
\[
\Omega_{\alpha} = \{(x,y_1, y_2, \ldots, y_{n-1}): 0 < x < 1, (y_1^2 + y_2^2 + \cdots + y_{n-1}^2)^{\frac{1}{2}} < x^{\alpha}\}.
\]
Then $\Omega_{\alpha}$ is an $L^s$-averaging domain if and only if
\[
(\alpha - 1)(s - n + 1) < n.
\]
\end{thm}

Note that if $0 \leq \alpha \leq 1$ then $\Omega_{\alpha}$ is a John domain and hence $L^s$-averaging for all $s \geq 1$.  At the same time, it is straightforward to check that the inequality is satisfied.  With these observations in mind, we restrict to $\alpha > 1$ for the remainder of this section.

The proof is broken into parts.  In the first part essential tubes are used to show that a cusp cannot be $L^s$-averaging if the given inequality is not satisfied.  In the second part, a generalized Whitney subdivision is used to show the converse.

\begin{proof}[Proof (part one)] 
Let $0 < a < b < 1$ and consider the tube centered along the $x$-axis with ends at $a$ and $b$, and its radius equal to $b^{\alpha}$.  This is an essential tube with $r = b^{\alpha}$, $l = b-a$, and $c = \left(\frac{a}{b}\right)^{\alpha n}$.

Using the sequence $a_j = \frac{1}{2^j}$, consider the essential tubes $T_j$ defined as above using $a = a_j$ and $b = a_{j-1}$  Then
\begin{align*}
    E_{T_j}
        &= \sum_{j = 3}^{\infty} \left(\frac{a_j}{a_{j-1}}\right)^{\alpha n} (a_{j-1})^{\alpha n}
        \left(\frac{a_{j-1} - a_j}
            {(a_{j-1})^{\alpha}}\right)^{s+1} \\
        &= 2^{-\alpha(s+1)} \sum_{j = 3}^{\infty} \left(2^{(\alpha - 1)(s-n+1)-n}\right)^j.
\end{align*}
If $(\alpha-1)(s-n+1) \geq n$ then this series diverges, so by Corollary \ref{tubeestcor}, $\Omega_{\alpha}$ cannot be $L^s$-averaging.
\end{proof}

For the other direction, we first establish some structure and initial results.  Because $\Omega_{\alpha}$ is symmetric about the $x$-axis, it is beneficial to work in cylindrical coordinates $(x, r, \theta)$.  In these coordinates,
\[
\Omega_{\alpha} = \{(x, r, \theta):0 < x < 1, r < x^{\alpha}, \theta \in \mathbb{S}^{n-2}\},
\]
and the volume element is
\[
dx\, dy_1 \cdots dy_{n-1} = r^{n-2} dx\, dr\, d\theta
\]
where $d \theta$ is the volume element for the unit $(n-2)$-sphere $\mathbb{S}^{n-2}$.

Given $j,m \in \mathbb{Z}^+$, let
\[
S_{j,m} = \left\{(x, r, \theta): \frac{m}{2^{j+\ell}} \leq x \leq \frac{m+1}{2^{j+\ell}}, \left(1 - \frac{1}{2^{\ell-1}}\right) x^{\alpha} \leq r \leq \left(1 - \frac{1}{2^{\ell}}\right) x^{\alpha} \right\}
\]
where $\ell = \lfloor\log_2(m)\rfloor + 1$.  See Figure \ref{cuspwhitneyfig}.

\begin{figure}
\begin{picture}(330,150)
\put(0,-2){
\includegraphics[scale = .7, clip = true, draft = false]{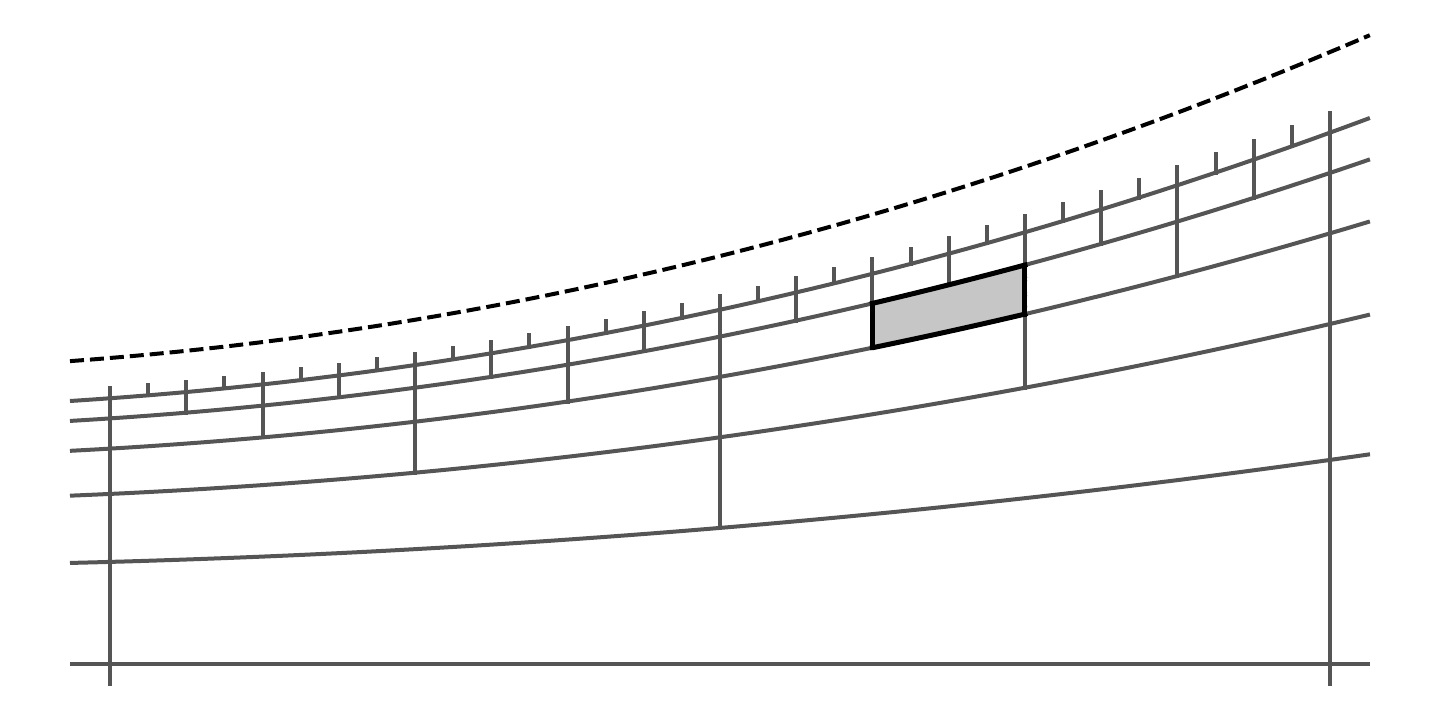}
}
\put(65,68){$\cdot$}
\put(65,70){$\cdot$}
\put(65,72){$\cdot$}
\put(120,77){$\cdot$}
\put(120,79){$\cdot$}
\put(120,81){$\cdot$}
\put(176,90){$\cdot$}
\put(176,93){$\cdot$}
\put(176,96){$\cdot$}
\put(231,108){$\cdot$}
\put(231,111){$\cdot$}
\put(231,114){$\cdot$}
\put(286,138){$r = x^{\alpha}$}
\put(286,99){$r = \frac{7}{8}x^{\alpha}$}
\put(286,80){$r = \frac{3}{4}x^{\alpha}$}
\put(286,51){$r = \frac{1}{2}x^{\alpha}$}
\put(286,8){$r = 0$}
\put(295,111){$\boldsymbol{\cdot}$}
\put(295,120){$\boldsymbol{\cdot}$}
\put(295,129){$\boldsymbol{\cdot}$}
\put(147,18){1}
\put(82,37){\rotatebox{4}{10}}
\put(206,50){\rotatebox{7}{11}}
\put(49,47){\rotatebox{4}{\small 100}}
\put(111,53){\rotatebox{6}{\small 101}}
\put(173,63){\rotatebox{8}{110}}
\put(235,75){\rotatebox{11}{111}}
\put(34,55){\rotatebox{5}{\tiny 1000}}
\put(65,58){\rotatebox{6}{\tiny 1001}}
\put(96,62){\rotatebox{7}{\tiny 1010}}
\put(127,66){\rotatebox{9}{\tiny 1011}}
\put(158,72){\rotatebox{10}{\tiny 1100}}
\put(187,78){\rotatebox{11}{\scriptsize 1101}}
\put(218,86){\rotatebox{12}{\scriptsize 1110}}
\put(249,95){\rotatebox{15}{\scriptsize 1111}}
\put(10,-3){$x = \frac{1}{2^{j+1}}$}
\put(255,-3){$x = \frac{1}{2^j}$}
\end{picture}
\caption{Part of the generalized Whitney subdivision of a Cusp Domain.  Note that in order to show detail, the curves representing the layers are not to scale.  The sets $S_{j,m}$ are labeled here with the binary representation of $m$.  Moving from one such set to the set below it, the label is truncated by removing the right-most digit.  For example, starting at $S_{j,13}$ (shaded), the sets below it have indices 6, 3, and 1.} \label{cuspwhitneyfig}
\end{figure}
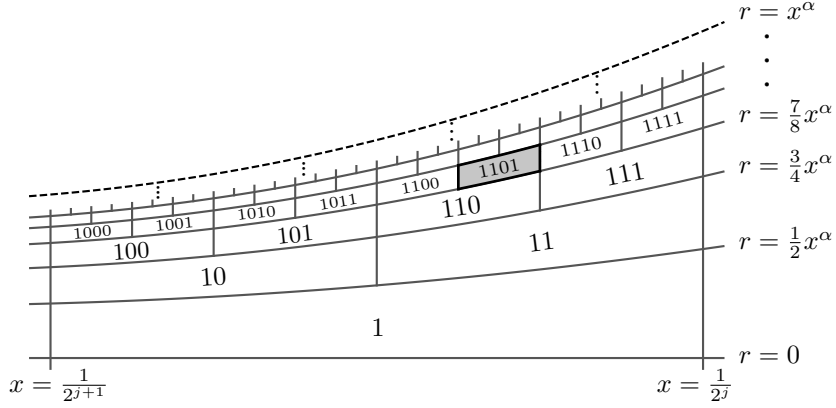

These sets are created by first dividing the domain into disks indexed by $j$, then layers indexed by $\ell$, and finally further subdividing into the sets described.  Since, in the third step, each layer is subdivided into twice as many sets as the previous layer, the number of digits in the base-2 representation of $m$ is the layer $\ell$ so that $m < 2^{\ell} \leq 2m$.

This subdivision is not quite a generalized Whitney subdivision for two different reasons.  First the union of all of the $S_{j,m}$ misses a significant portion of $\Omega$.  We define $S_0 = \Omega_{\alpha} \cap \left\{(x, r, \theta): x > \frac{1}{4}\right\}$.  This set is a John domain and we do not attempt to subdivide it.

Second, when $m > 1$, the sets $S_{j,m}$ are not star-shaped.  When $n = 2$, each $S_{j,m}$ is the disjoint union of two sets, one lying above the $x$-axis and one lying below.  To resolve this, for $m > 1$ let $S_{j,m}^+ = S_{j,m} \cap \{y_1 > 0\}$ and let $S_{j,m}^- = S_{j,m} \cap \{y_1 < 0\}$.

When $n > 2$, each $S_{j,m}$ has a solid ring shape, or disk-like shape if $m = 1$.  To formally make use of the observations about diameter in the previous section, we could further subdivide each $S_{j,m}$ into star-shaped regions through some subdivision of $\mathbb{S}^{n-2}$.  However, we will find that because of the choice of the path, each such region would contribute the same, so we keep $S_{j,m}$ as a single set.

To help with calculations later, we have the following:

\begin{lemma} \label{cuspdiameterlemma}
For each $j$, $m$, denote the radial thickness and horizontal width of $S_{j,m}$ by $d_r$ and $d_x$ respectively.  Then
\[
d_r < 2 d_x.
\]
\end{lemma}

\begin{proof}
The horizontal width is
\begin{align*}
d_x &= \frac{m+1}{2^{j + \ell}} - \frac{m}{2^{j + \ell}} \\
    &= \frac{1}{2^{j + \ell}}.
\end{align*}
The radial thickness for a given $x$ is
\begin{align*}
d_r(x) &= \left(1 - \frac{1}{2^\ell}\right)x^{\alpha}
   - \left(1 - \frac{1}{2^{\ell-1}}\right)x^{\alpha} \\
   &= \frac{x^{\alpha}}{2^\ell},
\end{align*}
and this quantity is maximized at the right end of $S_{j,m}$ at $x = \frac{m+1}{2^{j+\ell}}$, so
\[
d_r = \frac{\left(\frac{m+1}{2^{j+\ell}}\right)^{\alpha}}{2^\ell}.
\]
Using the fact that $m < 2^{\ell}$, the ratio of these distances is
\begin{align*}
\frac{d_r}{d_x}
    &= \frac{\frac{
\left(\frac{m+1}{2^{j+\ell}}\right)^{\alpha}}
{2^{\ell}}}
{\frac{1}{2^{j+\ell}}} \\
&< \left(\frac{2^{\ell}+1}{2^{\ell}}\right)^{\alpha}
\frac{1}{2^{(\alpha-1)j}} \\
&< \frac{2^{\alpha}}{2^{(\alpha-1)j}} \\
&< \frac{2^{\alpha}}{2^{\alpha-1}} \\
&= 2.
\end{align*}
\end{proof}

The curves used to estimate the quasihyperbolic distance for $\Omega_{\alpha}$ will incorporate only horizontal and radial directions.  Therefore, since in the previous section, the diameter of the set is used as a proxy for the length of a curve for a bound on $k_{\Omega}$, we can restrict attention to $d_r$ and $d_x$, and in light of the previous lemma, we may use $d_x(S_{j,m})$ in place of $d(S_{j,m})$ and we have
\[
d_x(S_{j,m}) = \frac{1}{2^{j+\ell}}.
\]

Next we focus on distance to the boundary.

\begin{lemma} \label{cuspdisttobdylemma}
For each set $S_{j,m}$,
\[
\delta(S_{j,m}) \geq C(\alpha) \frac{1}{2^{\alpha j+\ell}}
\]
\end{lemma}
\begin{proof}
We first consider the two-dimensional case.  Let $f(x) = x^{\alpha}$ define the boundary.  Let $z = (x,y) \in \Omega_{\alpha} \cap \left\{x \leq \frac{1}{2}\right\}$, and let $\delta(z)$ be its distance to the boundary.  Then since $f$ is convex and increasing, $\delta(z)$ is at least the distance to the tangent line at $(x, f(x))$, and the distance to this tangent line is bounded below by a multiple of the vertical distance $x^{\alpha} - y$.  This multiple, $C_1(\alpha)$ depends on $\alpha$ only, and is realized at $x = \frac{1}{2}$.  The general case is similar, due to rotational symmetry.

Focusing now on $S_{j,m}$, the points closest to $\partial(\Omega_{\alpha})$ are $\left(\frac{m}{2^{j+\ell}},\left(1-\frac{1}{2^\ell}\right)\left(\frac{m}{2^{j+\ell}}\right)^{\alpha},\theta\right)$.  If we restrict our attention to just the radial distance we find
\begin{align*}
\delta(S_{j,m})
    &\geq C_1(\alpha)
    \left[
    \left(\frac{m}{2^{j+\ell}}\right)^{\alpha}
    - \left(1-\frac{1}{2^\ell}\right)
        \left(\frac{m}{2^{j+\ell}}\right)^{\alpha}
        \right]\\
    &= C_1(\alpha) \frac{1}{2^\ell}
    \left(\frac{m}{2^{j+\ell}}\right)^{\alpha} \\
    & \geq C_1(\alpha) \frac{1}{2^\ell}
    \left(\frac{2^{\ell-1}}
        {2^{j+\ell}}\right)^{\alpha} \\
    &= C(\alpha) \frac{1}{2^{\alpha j+\ell}}.
\end{align*}
\end{proof}
With these estimates in hand, we now have the following lemma.

\begin{lemma}
Using the basepoint $z_0 = \left(\frac{1}{3},0,0\right)$, for any point $z \in S_{j,m}$,
\[
k(z, z_0; \Omega_\alpha) \leq C(\alpha)(1+ \ell)2^{(\alpha-1)j}.
\]
\end{lemma}

\begin{proof}
For any point $z = (x, r, \theta) \in S_{j,m}$, $j \geq 1$, define the L-shaped path $\gamma:[0,2] \longrightarrow \Omega_{\alpha}$ by
\begin{equation*}
\gamma(t) = 
\begin{cases}
\left(x\,t + \frac{1}{3}(1-t), 0, 0\right)&\ \mathrm{if}\ t \in [0,1] \\
(x, r(t-1), \theta)&\ \mathrm{if}\ t \in [1,2].
\end{cases}
\end{equation*}
For the first part we only need the horizontal width of each set and for the second part, we only need the radial thickness.

Now we can estimate $k_{\Omega_{\alpha}}$.  For the first step, we determine which sets intersect $\gamma$. Let $(x, r, \theta) \in S_{j,m}$.  For the initial leg from $z_0$ to $(x, 0, 0)$, we use the sets $S_{i, 1}$ for $i = 1, \ldots, j$.  For the second leg from $(x, 0, 0)$ to $(x, r, \theta)$ we need to determine which sets lie between $S_{j,m}$ and $S_{j, 1}$  There is one at each layer out to the layer containing $S_{j,m}$, and the specific sets $S_{j, \lambda}$ are determined as follows.  Express $m$ in binary.  Then the $\lambda$ values are represented in binary by truncating the binary representation of $m$ by successively removing the rightmost digit.  For example if $m = 51$ then the $\lambda$ to use are:
\begin{align*}
    51 &= 110011_{2} \\
    25 &= 11001_2 \\
    12 &= 1100_2 \\
    6 &= 110_2 \\
    3 &= 11_2 \\
    1 &= 1_2
\end{align*}
Let $\Lambda(j,m)$ be the set of indices corresponding to these sets lying below $S_{j,m}$ and note that $|\Lambda(j,m)| = \ell$.  With the specific sets through which $\gamma$ passes known, Lemma \ref{kinitialestlemma}, modified to account for the fact that only the radial or horizontal distances are needed, and then Lemmas \ref{cuspdiameterlemma} and \ref{cuspdisttobdylemma} are used to approximate $k_{\Omega_{\alpha}}$ as follows:
\begin{align*}
    k(z, z_0; \Omega_{\alpha})
        &\leq 2\sum_{i=1}^j \frac{d_x(S_{i,1})}{\delta(S_{i,1})}
            + 2\sum_{\lambda \in \Lambda(k,m)} \frac{d_r(S_{j,\lambda})}{\delta(S_{j,\lambda})} \\
        &\leq 2\sum_{i=1}^j \frac{d_x(S_{i,1})}{\delta(S_{i,1})}
            + 4\sum_{\lambda \in \Lambda(k,m)} \frac{d_x(S_{j,\lambda})}{\delta(S_{j,\lambda})} \\
        &\leq C_1(\alpha) \left(\sum_{i=1}^j \frac{2^{\alpha i+1}}{2^{i+1}}
            + \sum_{\lambda \in \Lambda(j,m)} \frac{2^{\alpha j+\ell}}{2^{j+\ell}}
            \right)\\
        &= C_1(\alpha) \left(\sum_{i=1}^j 2^{(\alpha-1) i}
            + \sum_{\lambda \in \Lambda(j,m)} 2^{(\alpha-1) j} \right) \\
        &\leq C(\alpha)\left( 2^{(\alpha-1)j} + \ell 2^{(\alpha-1) j}\right) \\
        &= C(\alpha)(1 + \ell)2^{(\alpha-1) j}
\end{align*}
where the first sum on the third-to-last line is approximated by a constant times the largest term.
\end{proof}

With $k(z, z_0; \Omega_{\alpha})$ approximated, the next step is to estimate the measure of $S_{j,m}$.  
\begin{lemma}
For each $S_{j,m}$,
\[
|S_{j,m}| \leq C(\alpha, n) \frac{1}{2^{j[\alpha(n-1)+1]}} \frac{1}{2^{2\ell}}.
\]
\end{lemma}
\begin{proof}
We have
\[
    |S_{j,m}| = \int_{\mathbb{S}^{n-2}} \int_{\frac{m}{2^{j+\ell}}}^{\frac{m+1}{2^{j+\ell}}}
    \int_{\left(1-\frac{1}{2^{\ell-1}}\right)x^{\alpha}}^{\left(1-\frac{1}{2^{\ell}}\right)x^{\alpha}}
    r^{n-2}\; dr dx d\theta
\]
The integral over the sphere just produces a dimensional constant.  For the other two integrals, the given functions are increasing, and so are approximated by $\int_a^b f(x)\, dx \leq f(b)(b-a)$ resulting in
\begin{align*}
|S_{j,m}| &\leq C_1(n) \int_{\frac{m}{2^{j+\ell}}}^{\frac{m+1}{2^{j+\ell}}}\left[\left(1-\frac{1}{2^{\ell}}\right)x^{\alpha}\right]^{n-2}x^{\alpha}\frac{1}{2^{\ell}}\, dx \\
    &= C_1(n) \frac{1}{2^{\ell}}\left(1-\frac{1}{2^{\ell}}\right)^{n-2} \int_{\frac{m}{2^{j+\ell}}}^{\frac{m+1}{2^{j+\ell}}}x^{\alpha(n-1)}\, dx \\
    &\leq C_1(n) \frac{1}{2^{\ell}}\left(1-\frac{1}{2^{\ell}}\right)^{n-2} \left(\frac{m+1}{2^{j+\ell}}\right)^{\alpha(n-1)} \frac{1}{2^{j+\ell}}. \\
\end{align*}
Since $m+1 \leq 2m$ and $1-\frac{1}{2^{\ell}} < 1$, this simplifies to
\[
|S_{j,m}| \leq C_1(n) \frac{1}{2^{j[\alpha(n-1)+1]}} \frac{(2m)^{\alpha(n-1)}}{2^{2\ell}2^{\ell \alpha(n-1)}}.
\]
The final estimate then follows from the fact that $m < 2^{\ell}$.
\end{proof}

We are now in position to complete the proof of Theorem \ref{cuspthm}.

\begin{proof}[Proof (part two)]

The domain $\Omega_{\alpha}$ can be subdivided into a John domain and the family of $S_{j,m}$ as follows
\begin{align*}
\int_{\Omega_{\alpha}} [k(z, z_0; \Omega_{\alpha})]^s \, dz
    &= \int_{\Omega_{\alpha} \cap \{x > \frac{1}{2}\}} [k(z, z_0, \Omega_{\alpha})]^s \, dz + \sum_{j = 1}^{\infty} \sum_{m = 1}^{\infty} \int_{S_{j,m}} [k(z, z_0; \Omega_{\alpha})]^s \, dz \\
    &< \int_{S_0} [k(z, z_0; \Omega_{\alpha})]^s \, dz + \sum_{j = 1}^{\infty} \sum_{m = 1}^{\infty} \int_{S_{j,m}} [k(z, z_0; \Omega_{\alpha})]^s \, dz \\
    &< \int_{S_0} [k(z, z_0; S_0)]^s \, dz + \sum_{j = 1}^{\infty} \sum_{m = 1}^{\infty} \int_{S_{j,m}} [k(z, z_0; \Omega_{\alpha})]^s \, dz
\end{align*}
The first integral on the right is finite since the domain is a John domain.

For the sum of integrals, the estimate for $k(z, z_0; \Omega_{\alpha})$ and $|\Omega_{\alpha}|$ are combined to estimate $\int_{\Omega_{\alpha}} k(z, z_0; \Omega_{\alpha})^s \, dz$ as follows:
\begin{align*}
    \sum_{j = 1}^{\infty} \sum_{m = 1}^{\infty} \int_{S_{j,m}} [k(z, z_0; \Omega_{\alpha})]^s \, dz
        &\leq C(\alpha,n) \sum_{j = 1}^{\infty} \sum_{m = 1}^{\infty} [(1 + \ell)2^{(\alpha-1) j}]^s
        \frac{1}{2^{j(\alpha(n-1)+1)}} \frac{1}{2^{2\ell}} \\
        &= C(\alpha,n) \sum_{j = 1}^{\infty} [2^{(\alpha-1) s - (\alpha(n-1)+1)}]^j
            \sum_{m = 1}^{\infty} \frac{(1+\ell)^s}{2^{2\ell}}
\end{align*}
For the sum over $m$, note that $2^{2 \ell} > m^2$ and $(1+ \ell) \leq 2 + \log_2(m)$, so 
\[
\sum_{m = 1}^{\infty} \frac{(1+\ell)^s}{2^{2\ell}}       \leq \sum_{m = 1}^{\infty}
    \frac{[2+\log_2(m)]^s}{m^2}
\]
which converges for all $s$.

For the sum over $j$, this converges if and only if
\[
(\alpha-1) s - (\alpha(n-1)+1) < 0
\]
which can be rearranged as $(\alpha-1) (s - n + 1) < n$.
\end{proof}

We end with some comments. 1)  First, if $n = 2$, the domain $\Omega_{\alpha}$ is a finite intersection of John domains. 2) For all $n$, $\alpha$, $\Omega_{\alpha}$ is star-shaped and therefore $p$-Poincar\'{e} for all $1 \leq p < \infty$.

\section{Block Domains} \label{blockdomainsec}

In this section we build a domain using blocks, and show, by combining the techniques above, for which $s$ it is $L^s$-averaging. 

Consider the domain $\Omega \in \mathbb{R}^n$ defined as follows:  Starting with a closed unit cube $\Omega_1$, perform a triadic subdivision of the top face and glue a closed cube $\Omega_2$ onto the middle.  Then, on the top face of $\Omega_2$ glue on a cube $\Omega_3$ that is the same size as $\Omega_2$.  Next, in a similar fashion, perform a triadic subdivision the top face of $\Omega_3$, glue a cube $\Omega_4$ onto the middle, and then extend with three more cubes, all the same size.  Continue this process, doubling the number of same-sized cubes in each step so that the cubes $\Omega_{2^j}, \ldots \Omega_{2^{j+1}-1}$ have edge length equal to $3^{-j}$.  Finally, take the interior of the infinite union.  See Figure \ref{blockdomainfig}.

\begin{figure}
\begin{picture}(180,170)
\put(0,-8){
\includegraphics[scale = .55, clip = true, draft = false]{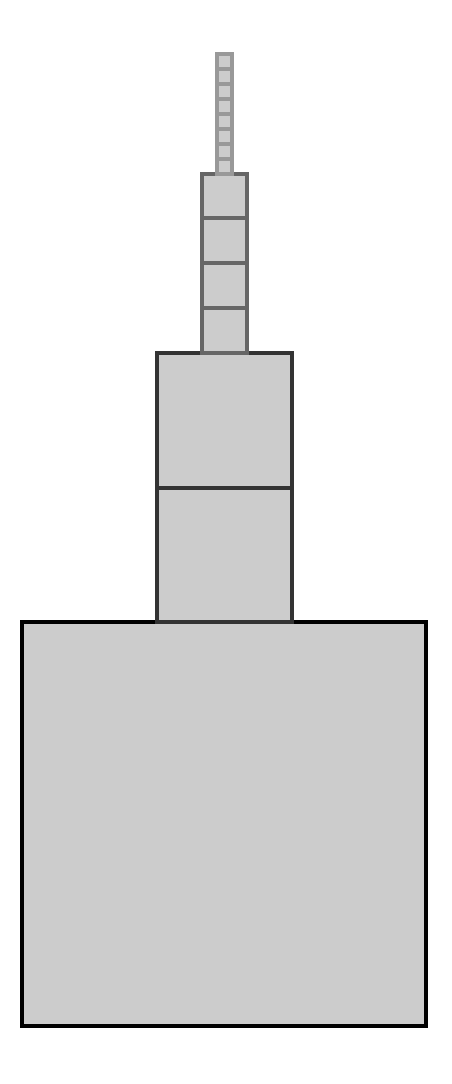}
}
\put(100,0){
\includegraphics[scale = .5, clip = true, draft = false]{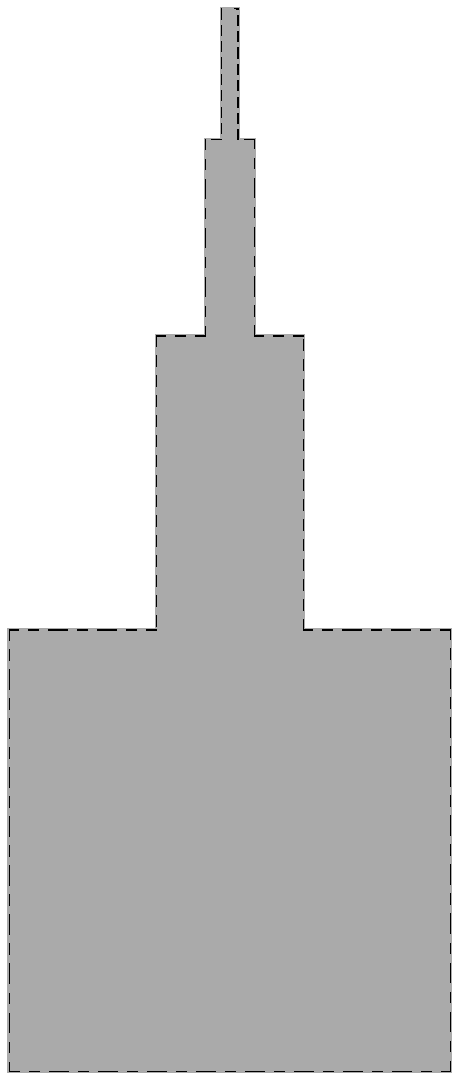}
}
\put(78,80){$\longrightarrow$}
\put(37.5,157){$\cdot$}
\put(37.5,160){$\cdot$}
\put(37.5,163){$\cdot$}
\put(135,157){$\cdot$}
\put(135,160){$\cdot$}
\put(135,163){$\cdot$}
\end{picture}
\caption{Block Domain Construction.} \label{blockdomainfig}
\end{figure}
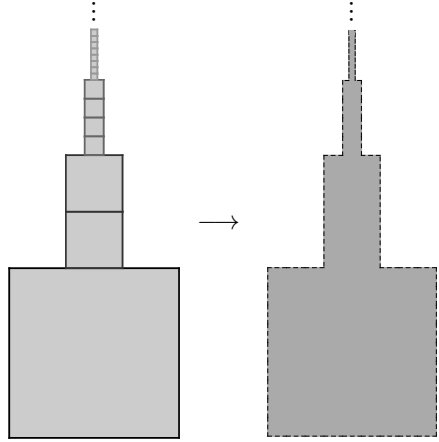

\begin{thm}
The set $\Omega$ is $L^s$-averaging if and only if $s < n\log_2(3) - 1$.
\end{thm}
As before, we separate the proof into two parts, beginning with the proof of when $\Omega$ fails to be $L^s$-averaging.
\begin{proof}[Proof (part one)]
Essential tubes can be built for each set of cubes $\Omega_{2^j}, \ldots \Omega_{2^{j+1}-1}$.  For the $j$th tube, $r_j = C_1(n) \left(\frac{1}{3}\right)^j$, $l_j = \left(\frac{2}{3}\right)^j$, and $c_j = C_2(n)$ so
\begin{align*}
E_{\{T_j\}} &= C(s, n)
     \sum_{j = 1}^{\infty}
        \left[\left(\frac{1}{3}\right)^j\right]^n
        \left(\frac{\frac{2}{3}}{\frac{1}{3}}\right)^{j(s+1)} \\
    &= \sum_{j = 1}^{\infty}
        \left[\frac{2^{s+1}}{3^n}\right]^j
\end{align*}
If $2^{s+1} \geq 3^n$ this sum diverges and so by Corollary \ref{tubeestcor}, $\Omega$ is not $L^s$-averaging, and this happens for $s \geq n\log_2(3) - 1$.
\end{proof}

\begin{proof}[Proof (part two)]
First, subdivide each $\Omega_m$ using the subdivision in the proof of Lemma~\ref{cubelemma}. This subdivision has the problem of needing infinitely many elements of the subdivision for any path connecting points in $\Omega_m$ to points in $\Omega_{m+1}$.  This is resolved by noting that when $\Omega_m$ and $\Omega_{m+1}$ are the same size, the sets at the centers of $\Omega_m$ and $\Omega_{m+1}$ can be connected with a third set of the same size, and when $\Omega_m$ and $\Omega_{m+1}$ are not the same size, the center set in $\Omega_{m+1}$ is the same size as the sets in the layer $L_1$ of $\Omega_m$ and can be connected to this layer by a single set of the same size.  See Figure \ref{adjustedsubdivisionfig}.

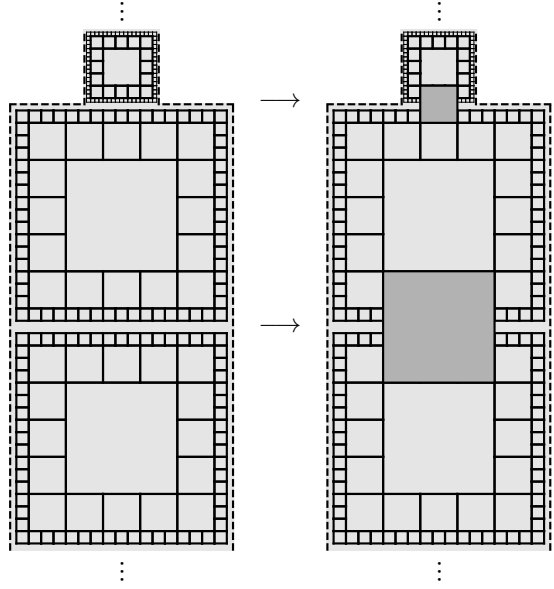
\begin{figure}
\begin{picture}(230,230)
\put(0,0){
\includegraphics[scale = .7, clip = true, draft = false]{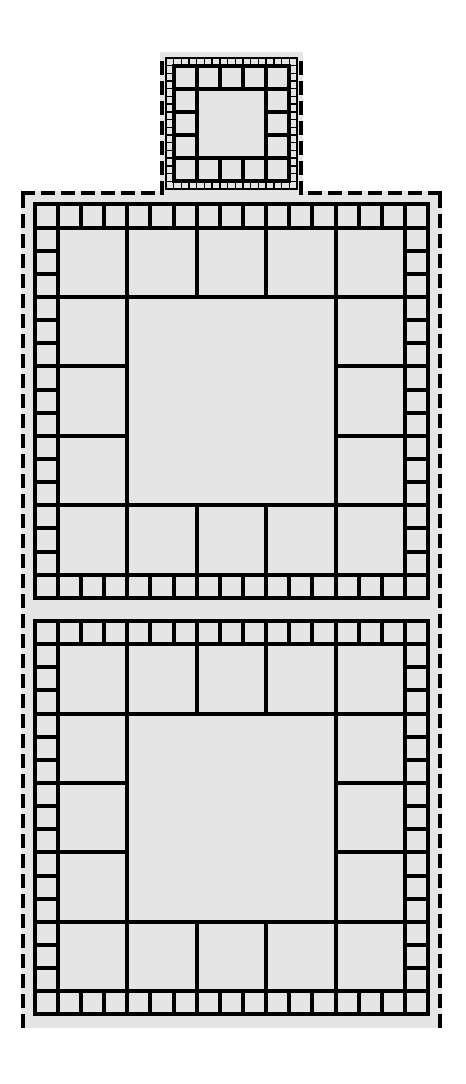}
}
\put(120,0){
\includegraphics[scale = .7, clip = true, draft = false]{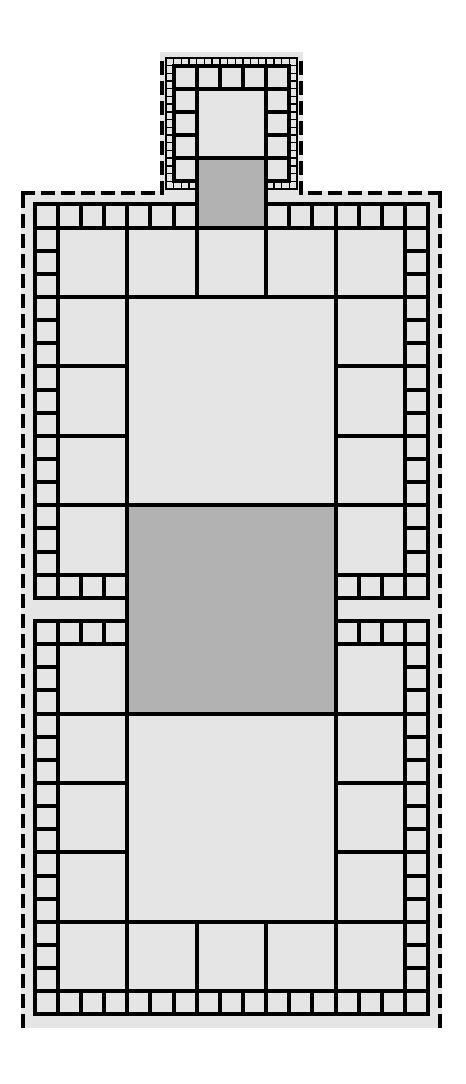}
}
\put(102,94){$\longrightarrow$}
\put(102,179){$\longrightarrow$}
\put(49,210){$\cdot$}
\put(49,213){$\cdot$}
\put(49,216){$\cdot$}
\put(49,4){$\cdot$}
\put(49,1){$\cdot$}
\put(49,-2){$\cdot$}
\put(169,210){$\cdot$}
\put(169,213){$\cdot$}
\put(169,216){$\cdot$}
\put(169,4){$\cdot$}
\put(169,1){$\cdot$}
\put(169,-2){$\cdot$}
\end{picture}
\caption{Modifying the subdivision.  The dark gray sets replace the sets they cover so as to connect the blocks.  Note that even with this modification, there are sets of measure zero at the boundary of each $\Omega_i$ that do not get covered.} \label{adjustedsubdivisionfig}
\end{figure}

Thus, to build a path from $z_0 \in \Omega_1$ to $z$ in the $i$th layer in $\Omega_m$, first walk to the center of $\Omega_m$, requiring at most $3m$ Whitney sets, and then to the $i$th layer, requiring at most $i+1$ additional steps, maybe much fewer if $z$ happens to be in or near one of the new big sets acting as a bridge into or out of $\Omega_m$.

Note that this generalized Whitney subdivision does not cover all of $\Omega$.  Namely, it misses most of the points at $\Omega_i \cap \Omega_{i+1}$.  This does not pose a difficulty though because it is a set of measure 0.

Combining this with the estimate in equation $\eqref{ksquareesteqn}$, and accounting for the sizes of the $\Omega_m$,
\[
\int_{\Omega}[k(z, z_0; {\Omega})]^s\, dz
    \leq C(n, s) \sum_{m=1}^{\infty} \sum_{i=0}^{\infty}(i+1+3m)^s \left(\frac{1}{3^n}\right)^{\lfloor \log_2(m)\rfloor} \left(\frac{1}{3}\right)^i
\]
where $\lfloor\log_2(m)\rfloor$ accounts for the size of the $\Omega_m$.  Noting that
$\lfloor\log_2(m)\rfloor \geq \log_2(m) - 1$,
it follows that
\begin{align*}
    \left(\frac{1}{3^n}\right)^{\lfloor\log_2(m)\rfloor}
    &\leq \left(\frac{1}{3^n}\right)^{\log_2(m) - 1} \\
    &=3^n \left(\frac{1}{3^n}\right)^{\log_2(m)} \\
    &= 3^n m^{-n\log_2(3)}.
\end{align*}
Plugging this into the estimate above yields
\[
\int_{\Omega}[k(z, z_0; {\Omega})]^s\, dz
    \leq C(n, s) \sum_{m=1}^{\infty} \sum_{i=0}^{\infty}(i+1+3m)^s m^{-n\log_2(3)} \left(\frac{1}{3}\right)^i.
\]
Note that $i+1+3m \geq m$ so if $s \geq n\log_2(3) - 1$ the double sum diverges and this estimate gives us no information.  On the other hand,
\begin{align*}
(i+1 + 3m)^s
    &= m^s\left(\frac{i+1}{m} + 3\right)^s \\
    &\leq m^s (i+4)^s
\end{align*}
so if $s < n\log_2(3) - 1$ then the double sum above converges and so $\Omega$ is $L^s$-averaging.
\end{proof}

\subsection{Variations}

With this initial tower in hand, there are a number of modifications that can be made without significantly changing the analysis.

First, we could glue towers of cubes on to all faces of the initial cube, and more generally, we could add other smaller towers as well.  As long as the number of additional towers is bounded, the estimates above will still hold.

We could consider more extreme ratios of side lengths of adjacent squares.  This will only affect the contribution of the number of steps to go from one center to the next.  As long as this stays bounded, the analysis above will hold.

We could glue the cubes together in different orientations to produce spirals, trees, or other interesting fractal shapes.  For the above analysis to hold, the key thing that would need to be preserved is that the number of cubes of a given size stay comparable to the number introduced above.  More exotic shapes could be considered with more careful analysis.

The ``2'' in the critical value comes from the growth in the number of cubes of a given size and the ``3'' comes from the ratio of one size to the next.  Playing with these values would produce other relationships.  In the current case, ``3'' was chosen because it was relatively easy to verify that the Whitney subdivision has the correct properties, and then ``2'' was the only available integer of any interest.  For example, using ``1'' instead of ``2'', we get something like an Aztec pyramid, which is $L^s$-averaging for all $s \geq 1$, and in fact is John.

\section {\bf The union of $L^s(\mu)$-averaging domains} \label{unionsec}

In 1999, Jussi Vaisala proved that, under appropriate conditions, the union of John domains is still a John domain in \cite{Vaisala}. Since $L^s(\mu)$-averaging domains are extensions of John domains, a natural question is:  Does the union of $L^s(\mu)$-averaging domains have the similar property? We will answer this question in this section.

We say a weight $w(z)$ satisfies the $A_r$ condition in a domain $\Omega$, and write $w\in A_r(\Omega)$,
$r>1$, if
\[
\sup_{B \subset \Omega} \left(\frac{1}{|B|} \int_B w\, dz\right)
\left(\frac{1}{|B|}\int_B w^{\frac{1}{1-r}}\, dz\right)^{r-1} < \infty.
\]
Note that if $w \in A_r(\Omega)$ and $G \subset \Omega$ then $W \in A_r(G)$ as well.  With this weight, the measure $\mu$ is defined by $d\mu = w(z) dz$

The following result, found in \cite{LiuDing}, gives a necessary and sufficient condition for a domain to be $L^s(\mu)$-averaging so long as the weight function defining $\mu$ satisfies the $A_r$ condition.

\begin{lemma} \label{lsavgweightlemma}
Let $w \in A_r$ for $r>1$ and $\mu$ be a measure defined by $d\mu = w(z)dz$.  Then $\Omega$ is an $L^s(\mu)$-averaging domain if and only if the inequality
\[
\left(\frac{1}{\mu(\Omega)}
    \int_{\Omega} k(z,z_0; \Omega)^s d\mu \right)^{\frac{1}{s}} \leq C
\]
holds for some fixed point $z_0$ in $\Omega$ and a constant $C$ depending only on $n$, $s$, $\mu(\Omega)$, the choice of $z_0 \in \Omega$, and the constant from the inequality in the definition of $L^s(\mu)$-averaging domains.
\end{lemma}

\begin{thm} \label{unionthm}
Let $G_1$ and $G_2$ be bounded $L^s(\mu)$-averaging domains with $G_1\cap G_2 \neq \emptyset$, where the measure $\mu$ is defined by $d\mu = w(z)dz$, and $w\in A_r(G_1 \cup G_2)$. Then, $G_1 \cup G_2$ is also an $L^s(\mu)$-averaging domain.
\end{thm}

\begin{proof}
First, we show that for any two domains $D$ and $G$ with $D \subset G$, we have
\begin{equation} \label{ksubsetesteq}
k(z, z_0; G) \leq k(z, z_0; D)
\end{equation}
for any $z, z_0$ in $D$. We know that for any $z \in D$, it follows that
\[
d(z, \partial G) \geq d(z, \partial D),
\]
so, for any rectifiable curve $\gamma$ in $D$ joining $z$ to $z_0$, we have
\[
\int_{\gamma} \frac{1}{d(\zeta, \partial G)} d\sigma \leq \int_{\gamma} \frac{1}{d(\zeta, \partial D)} d\sigma.
\]
Hence,
\[
\inf_{\gamma \subset D} \int_{\gamma} \frac{1}{d(\zeta, \partial G)} d\sigma \leq \inf_{\gamma \subset D} \int_{\gamma} \frac{1}{d(\zeta, \partial D)} d\sigma.
\]
Thus,
\begin{align*}
k(z, z_0; G)
    &= \inf_{\gamma \subset G} \int_{\gamma}
        \frac{1}{d(\zeta, \partial G)} d\sigma \\
    &\leq \inf_{\gamma \subset D} \int_{\gamma}
        \frac{1}{d(\zeta, \partial G)} d\sigma \\ 
    &\leq \inf_{\gamma \subset D} \int_{\gamma}
        \frac{1}{d(\zeta, \partial D)} d\sigma \\
    &= k(z, z_0; D).
\end{align*}
Now, choose $z_0\in G_1 \cap G_2$.  For $i = 1, 2$, extend the definitions of $k(z, z_0; G_i)$ to $G_1 \cup G_2$ by
\begin{equation*}
	k^*_i(z, z_0)=
	\begin{cases}
		k(z, z_0; G_i), & z \in G_i \\
		0, & z \not \in G_i.
	\end{cases}
\end{equation*}
Then, by Equation~\eqref{ksubsetesteq}, we have
\begin{equation} \label{kunionesteq}
k(z, z_0; G_1 \cup G_2)
    \leq k^*_1(z, z_0) + k^*_2(z, z_0).
\end{equation}
Since $G_1$ and $G_2$ are $L^s(\mu)$-averaging domains,
by Lemma~\ref{lsavgweightlemma}, for $i = 1, 2$ and $z_0 \in G_1 \cap G_2$ we have
\begin{equation} \label{g1intesteq}
\frac{1}{\mu(G_i)}\int_{G_i} k(z,z_0; G_i)^sd\mu \leq C_i.
\end{equation}
Using Equations~\eqref{kunionesteq} and \eqref{g1intesteq}, and the elementary inequality 
\[
(|a| + |b|)^s \leq 2^s (|a|^s + |b|^s)
\]
for any $s>0$, we obtain
\begin{align*}
\frac{1}{\mu(G_1 \cup G_2)} &\int_{G_1 \cup G_2}
                \left( k(z, z_0;G_1 \cup G_2) \right)^s d\mu \\
    &\leq \frac{1}{\mu(G_1 \cup G_2)} 
        \int_{G_1 \cup G_2} \left(
                k^*_1(z, z_0) + k^*_2(z, z_0) \right)^s d\mu \\
    &\leq \frac{1}{\mu(G_1 \cup G_2)}
        \int_{G_1 \cup G_2} 2^s \left(
            \left( ( k^*_1(z, z_0)\right)^s + \left ( k^*_2(z, z_0)
                \right)^s \right) d\mu \\
    &= \frac{2^s}{\mu(G_1 \cup G_2)}
        \int_{G_1 \cup G_2} \left ( k^*_1(z, z_0) \right)^s d\mu \\
    &\qquad \qquad + \frac{2^s}{\mu(G_1 \cup G_2)}
        \int_{G_1 \cup G_2} \left( k^*_2(z, z_0) \right)^s d\mu \\
    &\leq 2^s \left (\frac{1}{\mu(G_1)}
        \int_{G_1} \left ( k(z, z_0; G_1) \right)^s d\mu
            + \frac{1}{\mu(G_2)}
                \int_{G_2} \left( k(z, z_0; G_2) \right)^s d\mu \right) \\
&\leq 2^s (C_1 + C_2) \\
&= C_3
\end{align*}
which means that
\begin{equation} \label{unionintesteq}
\left (\frac{1}{\mu(G_1 \cup G_2)} \int_{G_1 \cup G_2}
    \left( k(z, z_0; G_1 \cup G_2) \right)^s d\mu
        \right)^{\frac{1}{s}} 
            \leq C_4.
\end{equation}
and hence, by Lemma~\ref{lsavgweightlemma} and Equation~\eqref{unionintesteq},
$G_1\cup G_2$ is an $L^s(\mu)$-averaging domain.
\end{proof}

Using Theorem~\ref{unionthm} and mathematical induction, we can prove the following theorem about the finite union of $L^s(\mu)$-averaging domains.

\begin{thm} \label{manyunionthm}
Let $w\in A_r(\cup_{i=1}^m G_i)$ and let $G_i$ be $L^s(\mu)$-averaging domains, $i=1,...,m$ such that
$\cup_{i=1}^m G_i$ is connected. Then,  $\cup_{i=1}^m G_i$ is also an $L^s(\mu)$-averaging domain.
\end{thm}

For any $t$ with $0 < t < s < \infty$ and any $z_0$ in a domain $G$, by H\"{o}lder's inequality
\[
\left( \int_G k(z, z_0; G)^t d\mu \right )^{\frac{1}{t}}
    \leq \left( \int_G k(z, z_0; G)^s d\mu \right )^{\frac{1}{s}}
            \left( \int_G d\mu \right)^{\frac{s-t}{st}}
\]
that is,
\begin{equation} \label{ltlsineq}
\left(
    \frac{1}{\mu(G)} \int_G k(z, z_0; G)^t d\mu \right)^{\frac{1}{t}}
    \leq \left(
        \frac{1}{\mu(G)}
                \int_G k(z, z_0; G)^s d\mu \right)^{\frac{1}{s}}.
\end{equation}
Applying Lemma~\ref{lsavgweightlemma} and Equation~\eqref{ltlsineq}, we have the following corollary immediately, which also appeared in \cite{LiuDing}.

\begin{cor} \label{ltlscor}
If $G$ is an $L^{s}(\mu)$-averaging domain, then $G$ is an $L^{t}(\mu)$-averaging domain for any $t$ with $0 < t < s$.
\end{cor}

From Theorem~\ref{manyunionthm} and Corollary~\ref{ltlscor}, we have the following result.

\begin{thm}
Let $w\in A_r(\cup_{i=1}^m G_i)$ and let $G_i$ be $L^{s_i}(\mu)$-averaging domains with $s_i > 0$, $i=1,...,m$ such that
$\cup_{i=1}^m G_i$ is connected. Then, $\cup_{i=1}^m G_i$ is also an $L^s(\mu)$-averaging domain, 
where $s = \min \{s_1, s_2, \cdots, s_m \}$.
\end{thm}

\bibliographystyle{amsalpha}
\bibliography{domains}

\end{document}